\documentclass{amsart}
\usepackage{amssymb}

\usepackage{upref}
\usepackage[T1]{fontenc}
\usepackage[utf8]{inputenc}
\usepackage{cleveref}

\newtheorem{theorem}{Theorem}[section]
\newtheorem{lemma}[theorem]{Lemma}
\newtheorem*{property*}{Property}
\newtheorem{proposition}[theorem]{Proposition}
\newtheorem{corollary}[theorem]{Corollary}

\theoremstyle{definition}
\newtheorem{definition}[theorem]{Definition}

\theoremstyle{remark}

\numberwithin{equation}{section}
\newcommand{\C}{\mathbb{C}}
\newcommand{\hol}{\mathcal{O}}
\newcommand{\XX}{\overline{X}}
\newcommand{\dive}{\operatorname{div}}
\newcommand{\Aut}{\operatorname{Aut}}
\newcommand{\VF}{\operatorname{VF}}
\newcommand{\CVF}{\operatorname{CVF}}
\newcommand{\SL}{\operatorname{SL}}
\newcommand{\Lie}{\operatorname{Lie}}
\newcommand{\Span}{\operatorname{Span}_\C}
\newcommand{\Ker}{\operatorname{Ker}}

\begin{document}
\title[Non-algebraic Manifolds with the VDP]{Non-algebraic Examples of Manifolds with the Volume Density Property}
\author[A. Ramos-Peon]{Alexandre Ramos-Peon}
\address{Mathematisches Institut\\Universität Bern\\Sidlerstr. 5\\3012 Bern\\Switzerland.}
\email{alexandre.ramos@math.unibe.ch}
\thanks{Partially supported by Schweizerischer Nationalfonds Grant 153120}

\subjclass[2010]{Primary 32M17,32H02. Secondary 32M25,14R10 }

\date{\today}

\commby{}

\begin{abstract}
Some Stein manifolds (with a volume form)  have a   large  group  of  (volume-preserving)  automorphisms: this is formalized by the (volume) density property, which has remarkable consequences.
Until now all known manifolds with the volume density property are algebraic, and the tools used to establish this property are algebraic in nature.
In this note we adapt a known criterion to the holomorphic case, and give the first known examples of non-algebraic manifolds with the volume density property: they arise as suspensions or pseudo-affine modifications 
over Stein manifolds satisfying some technical properties.
As an application we show that there are such manifolds that are potential counterexamples 
to the Zariski Cancellation Problem, a variant of the T\'{o}th-Varolin conjecture, and the problem of linearization of $\C^*$-actions on $\C^3$.
\end{abstract}

\maketitle
\section{Introduction}\label{sec:intro}
The group of automorphisms of complex affine space $\Aut(\C^n)$ has been intensively studied, both from the algebraic as from the analytic point of view. 
A foundational observation by E. Andersén and L. Lempert \cite{AL}, who proved that every polynomial vector field on $\C^n$ is a finite sum of complete vector fields, 
served as a starting point for new studies on $\Aut(\C^n)$ (because the flow of a complete field generates a $\C_+$-action on $\C^n$).
This led F. Forstneri\v{c} and J-P. Rosay \cite{FR} to the formulation which is now commonly called the Andersén-Lempert theorem: 
any local holomorphic flow defined near a holomorphically compact set can be approximated by global holomorphic automorphisms. 
Hence $\Aut(\C^n)$ is exceptionally large, and this result opens the possibility of constructing automorphisms with prescribed local behavior, with remarkable consequences, 
such as the existence of non-straightenable embeddings of $\C$ into $\C^2$ (see Section \ref{sec:last}), counterexamples to the holomorphic linearization problem \cite{Derksen-Kut}, among many others (see e.g. 
\cite{KK-state}).
This aspect of the study of the automorphism group may be referred to as Andersén-Lempert theory and is the subject of ongoing research.

In order to generalize those techniques to a wider class of manifolds, D. Varolin introduced in \cite{Varolin1} the concept of the \textit{density property}, which accurately captures the idea of a manifold having a ``large'' group of automorphisms. 
Examples include homogeneous spaces, Danilov-Gizatullin surfaces, as well as Danielewski surfaces (see below). Andersén considered even earlier, in \cite{A90}, the situation where the vector fields preserve the standard volume form on $\C^n$, obtaining similar results. 
There is a corresponding \textit{volume density property} for manifolds equipped with a volume form, which has been substantially less studied. Beyond $\C^n$, only a few isolated examples were  known to Varolin (see \cite{Varolin-sh}), including
$(\C^*)^n$ and $\SL_2(\C)$. It took around ten years until new instances of these manifolds were found in  \cite{KK-volume}: all linear algebraic groups equipped with the left invariant volume form, 
as well as some algebraic Danielewski surfaces (see \cite{KK-Corea} for an exhaustive list). 

In this note we exhibit new manifolds with the volume density property. We prove a general result (see Theorem \ref{th:4.1}), from which we can deduce the following:
\begin{theorem}\label{thm:i1}
Let $n\geq 1$ and $f\in\hol(\C^n)$ be a nonconstant holomorphic function with smooth reduced zero fiber $X_0$, such that $\tilde{H}^{n-2}(X_0)=0$ if $n\geq 2$.
Then the hypersurface $\overline{\C^n_f}=\{uv=f(z_1,\dots,z_n)\}\subset\C^{n+2}$ has the volume density property with respect to the form $\bar{\omega}$ satisfying $d(uv-f)\wedge\bar{\omega}=du\wedge dv \wedge dz_1\wedge\dots \wedge dz_n$.
\end{theorem}
For $n=1$ this manifold is called a \textit{Danielewski surface}.
Theorem \ref{thm:i1} was known in the special case where  $f$ is a polynomial: this is due to S. Kaliman and F. Kutzschebauch, see \cite{KK-volume}. Their proof  heavily depends on the use of Grothendieck's spectral sequence and seems difficult to generalize to the non-algebraic case.
Our method of proof is completely different. It relies on modifying and using a suitable criterion involving so-called semi-compatible pairs of vector fields, developed in \cite{KK-compatible} for the algebraic setting. This method will be explained in Section \ref{sec:first}.
In Section \ref{sec:second} we will study the suspension $\XX$ (or pseudo-affine modification) of rather general manifolds $X$ along $f\in\hol(X)$.
After some results concerning the topology and homogeneity of $\XX$, we will show that the structure of $\XX$ makes it possible to lift compatible pairs of vector fields from $X$ to $\XX$, in such a way that a technical but essential generating condition on $T\XX\wedge T\XX$ is guaranteed (Theorem \ref{Jo}).

It is  still unknown whether a contractible Stein manifold with the volume density property has to be biholomorphic to  $\C^n$. It is believed that the answer is negative, see
\cite{KK-volume}. For instance the affine algebraic submanifold of $\C^6$ given by the equation $uv = x+ x^2y + s^2 + t^3$ is such an example. 
Another prominent one is the Koras-Russell cubic threefold, see \cite{Leuenberger-volume}. 
In Section \ref{sec:last} we will show how to use Theorem \ref{th:4.1} to produce a non-algebraic manifold with the volume density property which is diffeomorphic to $\C^{n}$, which to our knowledge is the first of this kind.
In fact, we prove the following.
\begin{theorem}
Let $\phi:\C^{n-1}\hookrightarrow\C^{n}$ be a proper holomorphic embedding, and consider the manifold defined by $\overline{\C^n_f}=\{uv=f(z_1,\dots,z_n)\}\subset\C^{n+2}$, where $f\in\hol(\C^n)$ generates the ideal of functions vanishing on $\phi(\C^{n-1})$. Then $\overline{\C^n_f}$ is diffeomorphic to $\C^{n+1}$
and has the volume density property with respect to the volume form $\bar\omega$ satisfying $d(uv-f)\wedge\bar{\omega}=du\wedge dv \wedge dz_1\wedge\dots \wedge dz_n$.
Moreover $\overline{\C^n_f}\times \C$ is biholomorphic to $\C^{n+2}$, and therefore  is a potential counterexample to the Zariski Cancellation Problem if $\phi$ is not straightenable.
\end{theorem} 
We end Section \ref{sec:last} with two more examples which are related to the problem of linearization of holomorphic $\C^*$-actions on $\C^n$.

The author would like to thank F. Kutzschebauch for useful advice and suggestions during the preparation of the manuscript.
\section{A criterion for volume density property}\label{sec:first}
Let $X$ be a complex manifold of dimension $n$. 
We implicitly identify $T^{1,0}X$ with the real bundle $TX$; the global holomorphic sections of $TX$  are called holomorphic vector fields, 
and for simplicity we denote $\VF(X)$ the $\hol(X)$-module of all such fields.
Similarly, global holomorphic sections of the bundle $\wedge^j T^*X$  are called holomorphic $j$-forms, and we denote by $\Omega^j(X)$ the vector space of all such forms. 
In the sequel we drop the adjective \textit{holomorphic} since we only deal with such objects.
Of particular interest to us are the \textit{complete} vector fields on $X$: these are defined to be those whose flow, starting at any point in $X$, exists for all complex times, and hence generate one-parameter groups automorphisms of $X$.
We denote by $\CVF(X)$ the vector space of such fields, and note that given $\Theta\in\CVF(X)$, if either $f$ or $\Theta(f)$ lies in $\Ker(\Theta)$, then $f\Theta\in\CVF(X)$.
Observe that sums or Lie combinations of elements in $\CVF(X)$ are in general not complete; denote by $\Lie(X)$ the Lie algebra generated by the elements in $\CVF(X)$.

Assume $X$ is equipped with a \textit{volume form} $\omega$, that is, a non-degenerate $n$-form. 
Recall that the divergence of a vector field $\Theta$ on $X$ with respect to $\omega$ is the unique complex-valued function $\dive_\omega\Theta$ such that 
\[(\dive_\omega\Theta)\omega=\mathcal{L}_\Theta\omega\]
where $\mathcal{L}_\Theta$ is the Lie derivative in the direction of $\Theta$.
We can consider vector fields $\Theta$ of zero divergence with respect to $\omega$: 
$\mathcal{L}_\Theta\omega=0$, which  is equivalent to $\phi^*_t\omega=\omega$, where $\phi_t$ is the time $t$ map of the local flow of $\Theta$. 
Denote $\VF_\omega(X)$ the vector space of all such fields, which we also call \emph{volume-preserving} (note that this is not an $\hol(X)$-module anymore), 
and let $\mathcal{Z}^j(X)$ (resp. $\mathcal{B}^j(X)$) denote the vector space of $d$-closed (resp. $d$-exact) $j$ forms on $X$.
We denote by $\Lie_\omega(X)$ the Lie algebra generated by elements in $\CVF_\omega(X)=\VF_\omega(X)\cap\CVF(X)$.
The following is a definition of Varolin, making explicit the essential property of $\C^n$ necessary for the Andersén-Lempert behavior described in Section \ref{sec:intro}.
\begin{definition}
Let $X$ be a Stein manifold. We say that $X$ has the density property (in short DP) if $\Lie(X)$ is dense in $\VF(X)$ in the compact-open topology. 
If moreover $X$ is equipped with a volume form $\omega$, we say that $X$ has the volume density property with respect to $\omega$ (in short $\omega$-VDP) 
if $\Lie_\omega(X)$ is dense in $\VF_\omega(X)$.
\end{definition}
A manifold may have the VDP with respect to one form but not with respect to another one. It  does not imply in general the DP: take for instance $(\C,dz)$; less trivially,
$(\C^*)^k$ has the VDP but it is unknown if it has the DP for $k\geq 2$.
These definitions can be modified to the algebraic setting:
if we consider only algebraic vector fields on an affine algebraic variety, (and an algebraic volume form, respectively), and 
replace density by equality, the definitions above are that of the algebraic \mbox{($\omega$-V)DP}. 
It should be noted that the algebraic DP implies the DP as defined above; similarly, the algebraic VDP implies the holomorphic VDP, although
this is not a trivial fact (see \cite[Prop. 4.1]{KK-volume}).

An effective criterion for the algebraic density property was found by Kaliman and Kutzschebauch in \cite{KK-Criteria}. 
The idea is to find a nonzero $\C[X]$-module in $\Lie_{alg}(X)$, which can be ``enlarged''
in the presence of a certain homogeneity condition to the whole $\VF_{alg}(X)$. 
The module can be found as soon as there is a pair of complete fields which is ``compatible'' in a certain technical sense.
The algebraic VDP was first thoroughly studied in \cite{KK-volume}, 
and a corresponding criterion for was subsequently developed in \cite{KK-compatible}, wherein the notion 
of ``semi-compatible'' vector fields is central. 
In what follows, we give a holomorphic version of this criterion.

Given a vector field $\Theta\in \VF(X)$, there is a degree $-1$ $\wedge$-antiderivation $\iota_\Theta$ on the graded algebra of forms $\Omega(X)$
called \textit{interior product}, defined by the relation 
\[
(\iota_\Theta\alpha)(\nu)=\alpha(\Theta\wedge\nu),\quad  \alpha\in\Omega^{j+1}(X),\nu\in\Gamma(\wedge^kTX,X).
\]
Its relationship to the exterior derivative $d$ is expressed through Cartan's formula
\[
\mathcal{L}_\Theta\alpha=d\iota_\Theta\alpha+\iota_\Theta d\alpha.
\]
If $\omega$ is a volume form on $X$, non-degeneracy implies
that vector fields and $(n-1)$-forms are in one-to-one correspondence via $ \Theta\mapsto \iota_{\Theta}\omega,$
which by Cartan's formula restricts to an isomorphism
\[
\Phi:\VF_\omega(X)\to \mathcal{Z}^{n-1}(X).
\]
In the same spirit, there is an isomorphism of $\hol(X)$-modules 
\begin{equation}\label{intprod}
\Psi:\VF(X)\wedge \VF(X)\to \Omega^{n-2}(X),\quad \nu\wedge\mu\mapsto \iota_{\nu}\iota_\mu\omega 
\end{equation}
and it is straightforward that $\iota_\mu\iota_\nu\omega=\iota_{\nu\wedge\mu}\omega$.
We can deduce from the easily verified relation $[\mathcal{L}_\nu,\iota_\mu]=\iota_{[\nu,\mu]}$ that for $\nu,\mu\in\VF_\omega(X)$,
\begin{equation}\label{bracket}
\iota_{[\nu,\mu]}\omega=d\iota_{\nu}\iota_\mu\omega. 
\end{equation}
Hence by restricting the isomorphism in Equation \ref{intprod} to $\wedge^2\CVF_\omega(X)$ 
and composing with the exterior differential $d:\Omega^{n-2}\to\mathcal{B}^{n-1}$ we obtain a mapping
\[
d\circ\Psi:\CVF_\omega(X)\wedge\CVF_\omega(X)\to\mathcal{B}^{n-1},\quad \nu\wedge\mu\mapsto i_{[\mu,\nu]}\omega,
\]
whose image is in fact contained in $\Phi(\Lie_\omega(X))$. 

Suppose we want to approximate $\Theta\in\VF_\omega(X)$ on $K\subset X$ by a Lie combination of elements in $\CVF_\omega(X)$.
Consider the closed form $\iota_\Theta\omega$ and assume for the time being that it is exact. Then
by Equation $\ref{intprod}$ there is $\gamma\in\VF_\omega(X)\wedge\VF_\omega(X)$ such that
$\iota_\Theta\omega=d(\Psi(\gamma))$. 
It now suffices to approximate $\gamma$ by a sum of the form $\sum \alpha_i\wedge\beta_i\in\Lie_\omega(X)\wedge\Lie_\omega(X)$.
Indeed, by Equation \ref{bracket}, $\iota_\Theta\omega=d\circ\Psi(\gamma)$ would then be approximated by
elements \[
d\circ\Psi(\sum\alpha_i\wedge\beta_i)=\sum\iota_{[\alpha_i,\beta_i]}\omega\in\Phi(\Lie_\omega(X)),
\]which implies that $\Theta$ is approximated uniformly on $K$ by elements of the form 
$\sum[\alpha_i,\beta_i]\in\Lie_\omega(X)$, as desired. We therefore concentrate on this approximation on $\VF_\omega\wedge\VF_\omega(X)$. 
We will assume that \textit{(a)} there are $\nu_1,\dots,\nu_k,\mu_1,\dots,\mu_k\in\CVF_\omega(X)$ 
such that the submodule of $\VF(X)\wedge\VF(X)$ generated by the elements $\nu_j\wedge\mu_j$ is contained in the closure of $\Lie_\omega(X)\wedge\Lie_\omega(X)$. 
We may assume $K$ to be $\hol(X)$-convex, and let us suppose \textit{(b)} that for all $p$ in a Runge Stein neighborhood $U$ of $K$, 
$\nu_j(p)\wedge\mu_j(p)$ generate the vector space $T_pX\wedge T_pX$. We then proceed with standard methods in sheaf cohomology: let $\mathfrak{F}$ be the coherent sheaf corresponding to the wedge of the tangent bundle. Condition \textit{(b)} translates 
to the fact that 
the images of $\nu_j\wedge\mu_j$ generate the fibers of the sheaf, so by Nakayama's Lemma the lift to a set of generators for the stalks $\mathfrak{F}_p$ for all $p\in U$.
Therefore, by Cartan's Theorem B,
the sections of $\mathfrak{F}$ on $U$ are of the form \begin{equation}\label{wedges}
\sum h_i (\nu_j\wedge\mu_j)\end{equation}
with $h_j\in\hol(U)$.
Since $U$ is Runge, we conclude that every element $\gamma\in\VF_\omega(X)\wedge\VF_\omega(X)$
may be uniformly approximated on $K$ by elements as in Equation \ref{wedges} with $h_j\in\hol(X)$. By assumption \textit{(a)} $\gamma$ may be approximated uniformly on $K$ by elements in $\Lie_\omega(X)\wedge\Lie_\omega(X)$.

To find the pairs $\nu_j\wedge\mu_j$, observe that if $\nu,\mu\in\CVF_\omega(X)$, and $f\in\Ker\nu,g\in\Ker\mu$, then $f\nu,g\mu\in \CVF_\omega(X)$. 
By linearity, any element in the span of $(\Ker\nu\cdot\Ker\mu)\cdot(\nu\wedge\mu)$ lies in $\Lie_\omega(X)\wedge\Lie_\omega(X)$. 
By considering the closures, we see that if  $I$ is a nonzero ideal contained in the closure of $\Span{(\Ker\nu\cdot\Ker\mu)}$, 
then $I\cdot(\nu\wedge\mu)$ generates a submodule of $\VF(X)\wedge\VF(X)$ which is contained in $\overline{\Lie_\omega(X)\wedge\Lie_\omega(X)}$.
 This motivates the following definition.
\begin{definition}
Let $\nu,\mu$ be nontrivial complete vector fields on $X$. We say that the pair $(\nu,\mu)$ is \emph{semi-compatible} 
if the closure of the span of $\Ker \nu\cdot\Ker\mu$ contains a nonzero ideal of $\hol(X)$. 
We call the largest ideal $I\subset \overline{\Span{(\Ker\nu\cdot\Ker\mu)}}$ the ideal of the pair $(\nu,\mu)$.
\end{definition}
To reduce to the special case just treated (where $\iota_\Theta\omega$ is exact), we must further assume that given $\Theta\in\VF_\omega(X)$, it is possible to obtain the zero class in $H^{n-1}(X)$
by subtracting an element of  $\Phi(\Lie_\omega(X))$; however, Equation \ref{bracket} implies that Lie brackets represent the zero class in $H^ {n-1}(X)$, so it is enough to subtract elements from $\Phi(\CVF_\omega(X))$.
The preceding discussion then shows that the existence of ``enough'' semi-compatible pairs of volume-preserving vector fields, along with this condition, suffices to establish the VDP. We have thus proved the following criterion:
\begin{proposition}\label{crit}
Let $X$ be a Stein manifold of dimension $n$ with a holomorphic volume form $\omega$, satisfying the following condition:
\[
\text{every class of }H^{n-1}(X) \text{ contains an element in the closure of } \Phi(\CVF_\omega(X))
\]Suppose there are finitely many semi-compatible pairs of volume-preserving vector fields $(\nu_j,\mu_j)$ with ideals $I_j$
such that for all $x\in X$,
\[\left\{I_j(x)(\nu_j(x)\wedge \mu_j(x))\right\}_j\text{ generates }\wedge^2 T_xX.\]
Then $X$ has the $\omega$-VDP.
\end{proposition}
It is also possible to adapt the criterion for the algebraic DP of \cite{KK-Criteria}, using so-called compatible vector fields, which satisfy a stronger condition. Namely, a semi-compatible pair of (complete) vector fields
$(\nu,\mu)$ is called \emph{compatible} if there exists $h\in\hol(X)$ such that 
$\nu(h)\in\Ker\nu$ and $h\in\Ker\mu$.
Let $f\in\Ker\nu$ and $g\in\Ker \mu$. Then $f\nu,fh\nu,g\mu,gh\mu$ are complete vector fields and a simple calculation shows that
\[
fg\nu(h)\mu=[f\nu,gh\mu]-[fh\nu,g\mu]\in\Lie(X).
\]
In other words, if $I$ is the ideal associated to the pair  $(\nu,\mu)$, then $I\cdot \nu(h)\cdot \mu$ generates a submodule of $\VF(X)$ which is contained in $\Lie(X)$.
So by an obvious variant of the above discussion, we obtain the following generalized criterion for the DP.
\begin{proposition}
Let $X$ be a Stein manifold. Suppose there are finitely many compatible pairs of vector fields  $(\nu_j,\mu_j)$ such that $I_j(x)\nu_j(h_j(x))$ generate $T_xX$ for all $x\in X$. 
Then $X$ has the DP.
\end{proposition}
\section{Suspensions}\label{sec:second}
Let $X$ be a connected Stein manifold of dimension $n$, and let $f\in\hol(X)$ be a nonconstant holomorphic function with a smooth reduced zero fiber $X_0$ (this means that $df$ is not identically $0$ on $X_0$).
To it we associate the space  $\overline{X}$, called the \textit{suspension}  over $X$ along $f$, 
which is defined as 
\[
\overline{X}=\{  (u,v,x)\in\C^2\times X ; uv-f(x)=0 \}.
\]
Since  $X_0$ is reduced, $d(uv-f)\neq 0$ everywhere, so $\XX$ is smooth. Hence $\XX$ is a Stein manifold of dimension $n+1$.

Suppose $X$ has a volume form $\omega$. Then $\Omega=du\wedge dv\wedge \omega$ is a volume form on $\C^2\times X$. 
There exists a canonical volume form $\overline{\omega}$ on $\overline{X}$ such that
\[
d(uv-f)\wedge \overline{\omega}=\Omega|_{\overline{X}}.
\]
Moreover, any vector field $\overline{\Theta}$ on $\overline{X}$ has an extension $\Theta$ to $\C^2\times X$ with $\Theta(uv-f)=0$, and we 
have $\dive_{\overline{\omega}}\overline{\Theta}=\dive_\Omega\Theta|_{\overline{X}}$ (see \cite[2.2,2.4]{KK-Zeit}).
In view of our criterion we now investigate the existence of sufficient semi-compatible fields, as well as the topology of $\overline{X}$.

Let $\Theta\in\VF(X)$. There exists an extension $\tilde{\Theta}\in\VF(\C^2\times X)$ such that $\tilde{\Theta}(u)=\tilde{\Theta}(v)=0$ and $\tilde{\Theta}(\tilde{g})=\Theta(g)$ for all $g\in\hol(X)$ (here $\tilde{g}$ is an extension of $g$ not depending on $u,v$). Clearly,
$\dive_\Omega\tilde{\Theta}=\pi^*(\dive_\omega\Theta)$, where $\pi:\C^2\times X\to X$ is the natural projection.
We may ``lift'' $\Theta$ to a field in $\overline{X}$ in two different ways. Consider the fields on $\C^2\times X$
\[\Theta_u=v\cdot \tilde{\Theta}+\overset{\sim}{\Theta(f)}\frac{\partial}{\partial u} \quad\quad \Theta_v=u\cdot \tilde{\Theta}+\overset{\sim}{\Theta(f)}\frac{\partial}{\partial v},\]
which are clearly tangent to $\XX$; we may therefore consider  the corresponding fields (restrictions) on $\overline{X}$, which we denote simply ${\Theta_u}$ and ${\Theta_v}$.
\begin{lemma}\label{l1}
If $\Theta$ is $\omega$-volume-preserving, then ${\Theta_u}$ and ${\Theta_v}$ are of $\overline{\omega}$-divergence zero. 
Moreover, if $\Theta$ is complete, then ${\Theta_u}$ and ${\Theta_v}$ are also complete.
\end{lemma}
\begin{proof}
The completeness of the lifts is clear, but it will be useful for the sequel to compute explicitly their flows. Denote by $\phi^t(x)$ the flow of $\Theta$ on $X$, and let $g:X\times \C\to \C$ be the first order approximation of $f\circ \phi$ with respect to $t$;
in other words, let $g$ satisfy \begin{equation}\label{defining}
f(\phi^t(x))=f(x)+tg(x,t).                                 
                                \end{equation}
Since $f$ is holomorphic, $g$ is well defined and holomorphic on $X\times \C$. The claim is that $\Phi:\XX\times \C_t\to \XX $ defined by
\begin{equation}\label{flowoflift}
\Phi^t(u,v,x)=(u+tg(x,tv),v,\phi^{tv}(x))
\end{equation}
is the flow of $\Theta_u$, which therefore exists for all $t$. Indeed, we compute\[
\Theta_u(\Phi^t(u,v,x))=v\cdot\Theta({\phi^{tv}(x))}+\Theta (f)(\phi^{tv}(x))\frac{\partial}{\partial u},
\]
while on the other hand \[\frac{\partial}{\partial t}\Phi^t(u,v,x)=\frac{\partial}{\partial t}(tg(x,tv))\frac{\partial}{\partial u}+\frac{\partial}{\partial t}(\phi^{tv}(x))=\frac{\partial}{\partial t}(tg(x,tv))\frac{\partial}{\partial u}+v\cdot\Theta(\phi^{tv}(x)).
                        \]
The equality $\frac{\partial}{\partial t}(tg(x,tv))=\Theta (f)(\phi^{tv}(x))$ follows by differentiating Equation \ref{defining} at $(x,tv)$.

Since $\dive_{\overline{\omega}}{\Theta_u}=\dive_\Omega\Theta_u|_{\overline{X}}$, and because divergence (with respect to any volume form) is linear
and satisfies $\dive(h\cdot\Theta)=h\dive\Theta+\Theta(h)$, we get
\[
\dive_{\overline{\omega}}{\Theta_u}=
v\cdot \dive_{\Omega}\tilde{\Theta}|_{\overline{X}}+\tilde{\Theta}(v)+\Theta(f)\dive_\Omega\left(\frac{\partial}{\partial u}\right)+\frac{\partial}{\partial u}(\Theta(f))
=v\cdot \dive_{\Omega}\tilde{\Theta}|_{\overline{X}}\]
and as noted above $\dive_{\Omega}\tilde{\Theta}=\pi^*(\dive_\omega\Theta)=0$.
\end{proof}

\begin{lemma}\label{l2}
 Suppose $(\nu,\mu)$ is a semi-compatible pair of vector fields on $X$. 
 Then $({\nu_u},{\mu_v})$ and $({\nu_v},{\mu_u})$ are semi-compatible pairs on $\overline{X}$.
\end{lemma}
\begin{proof}
By Lemma \ref{l1}, the lifted and extended fields are complete. It then suffices to show that $(\nu_u,\mu_v)$ is a semi-compatible pair in $\C^2\times X$, 
because we may restrict the elements in the ideal to $\XX$:  by the Cartan extension theorem, this set forms an ideal in $\hol(\XX)$.

Let $I$ be the ideal of the pair $(\nu,\mu)$. For any function $h\in\hol(X)$, denote \[
\tilde{I}=\left\{ \tilde{h}\cdot F(u,v);h\in I,F\in\hol(\C^2)\right\}\subset\hol(\C^2\times X),                                   \]                                                           
where $\tilde{h}$ is the trivial extension as above. This is clearly a nonzero ideal. An element in $\tilde{I}$ can be approximated uniformly on a given compact of $\C^2\times X$ by a finite sum
\[
(\sum_k \tilde{n}_k\tilde{m}_k)\sum_{i,j}a_{i,j}u^i v^j=
\sum a_{i,j}(\tilde{n}_k v^j)(\tilde{m}_k u^i)
\]
where $n_{k}\in\Ker(\nu),m_k\in\Ker(\mu)$ for all $k$.
Since $\tilde{n}_k v^j\in\Ker(\nu_ u)$ for all $j,k\geq 0$ and $\tilde{m}_k u^i \in\Ker(\mu_v)$ for all $i,k\geq 0$,
it follows that $\tilde{I}$ is contained in the closure of $\Span(\Ker(\nu_u)\cdot\Ker(\mu_v))$.
\end{proof}
The topology of the suspension $\XX$  is of course closely related to that of $X$. In the case where $X$ is the affine space, this relationship is computed in detail in \cite[\S 4]{KZaffine}.
For more general $X$ we have the following.
\begin{proposition}\label{P2}
Assume $X$ has dimension $n\geq 2$. If the complex de Rham cohomology groups satisfy $H^n(X)=H^{n-1}(X)=0$ and $\tilde{H}^{n-2}(X_0)=0$, where $\tilde{H}$ denotes reduced cohomology, 
then $H^n(\overline{X})=0$.
\end{proposition}
\begin{proof}
 Consider the long exact sequence of the pair $(\XX,\XX\setminus U_0)$ in cohomology, where $U_0$ is the subspace of $\XX$ where $u$ vanishes:
\begin{equation}\label{exact}
\dots\to H^n(\XX,\XX\setminus U_0)\to H^n(\XX)\to H^n(\XX\setminus U_0)\to \dots
\end{equation}
The term on the right vanishes, because  $\XX\setminus U_0$ is biholomorphic to $\C^*\times X$ via $(u,x)\mapsto (u,f(x)/u,x)$, so 
\[H^n(\XX\setminus U_0)=(H^1(\C^*)\otimes H^{n-1}(X))\oplus (H^0(\C^*)\otimes H^n(X))=0.
\]
To evaluate the left-hand side, we use an idea due to Zaidenberg (see \cite{Zexotic}). Consider the normal bundle $\pi:N\to U_0$ of the closed submanifold $U_0$ in $\XX$, with zero section $N_0\cong U_0$.
Fix a tubular neighborhood $W$ of $U_0$ in $\XX$ such that the pair $(W,U_0)$ is diffeomorphic to $(N,N_0)$. 
Then by excision, we have that \[
\tilde{H}^*(\XX,\XX\setminus U_0)\cong \tilde{H}^*(W,W\setminus U_0)\cong \tilde{H}^*(N,N\setminus N_0).                                
                               \]
Let $t\in H^2(N,N\setminus N_0)$ be the Thom class of $U_0$ in $\XX$, that is, the unique cohomology class taking value $1$ on any oriented relative $2$-cycle in $H_2(N,N \setminus N_0)$ defined by a fiber $F$ of the normal bundle $N$ (see e.g.  \cite[\S 9--10]{Charclass}, for details). 
Then, by taking the cup-product of the pullback under $\pi$ of a cohomology class with $t$, we obtain the Thom isomorphisms
\[
H^i(U_0)\cong H^{i+2}(N,N\setminus N_0)\cong H^{i+2} (\XX,\XX\setminus U_0)\quad\forall i.
\]
Since $U_0\cong X_0\times \C$, $U_0$ is homotopy equivalent to $X_0$, and we have
$H^n(\XX,\XX\setminus U_0)\cong H^{n-2}(X_0)$.
If $n\geq 3$, reduced cohomology coincides with standard cohomology, and therefore $H^n(\XX)=0$ by exactness of Equation \ref{exact}. If $n=2$, that sequence becomes
\[
\dots\to H^1(\XX\setminus U_0)\to H^2(\XX,\XX\setminus U_0)\to H^2(\XX)\to 0.
\]
Let $\gamma$ be an oriented $2$-cycle  in $\XX$ whose boundary $\partial \gamma$  lies in $\XX\setminus U_0$ (a disk transversal to $U_0$). 
A one-dimensional subspace of $H^1(\XX\setminus U_0)$ is generated by a $1$-cocycle taking value $1$ on $\partial \gamma$, and this cocycle is sent via the coboundary operator (which is the first map in the above sequence)
to a $2$-cocycle taking value $1$ on $\gamma$, i.e., to the Thom class $t$ described previously, which is also a generator of a one-dimensional subspace of $ H^2(\XX,\XX\setminus U_0)$. However, $H^1(\XX\setminus U_0)\cong H^1(\C^*\times X)\cong \C$
and $H^2(\XX,\XX\setminus U_0)\cong H^0(U_0)\cong \C$, so the coboundary map is an isomorphism, and by exactness it follows that $H^2(\XX)=0$.
\end{proof}

Next, we show how to lift a collection of semi-compatible fields to the suspension and span $\wedge^2 T\XX$ with semi-compatible fields\footnote{A simpler algebraic case has been treated by J. Josi (Master thesis, 2013, unpublished)}. 
We will denote by $\Aut(X)$ (resp. $\Aut_\omega(X))$ the group of holomorphic automorphisms of the manifold $X$ (resp. the volume-preserving automorphisms). 
\begin{theorem}\label{Jo}
Let  $X$ be a Stein manifold with a finite collection $S$ of semi-compatible pairs $(\alpha,\beta)$ of vector fields such that for some $x_0\in X$ \begin{equation}\label{span}
\{\alpha(x_0)\wedge\beta(x_0);(\alpha,\beta)\in S\}\text{  spans } \wedge^2(T_{x_0} X).                                                
                                                                                                                                          \end{equation}
Assume that $\Aut(\XX)$ acts transitively on $\overline{X}$. Then there exists a finite collection $\overline{S}$ of semi-compatible pairs $(A_j,B_j)$ on $\overline{X}$  with corresponding ideals $I_j$  such that 
\begin{equation}\label{spans}
\Span\{I_j(\bar x) A_j(\bar{x})\wedge B_j(\bar{x})\}_j=\wedge^2(T_{\bar{x}} \overline{X})\quad \forall \bar{x}\in\overline{X}.                                                                                                                                                                                  
                \end{equation}
                Moreover, if $X$ has a volume form $\omega$ and the fields in $S$ preserve it, and $\Aut_{\bar\omega}(\XX)$ acts transitively, then the fields in $\overline{S}$ can be chosen to preserve the form $\bar{\omega}$                                                                                                                                                                               
\end{theorem}
\begin{proof}
We claim that it is sufficient to show that the conclusion holds for a single $\bar{x}_0\in\XX$.
Indeed, let $C$ be the analytic set of points $\bar{x}\in \XX$ where Equation \ref{spans} does not hold, and decompose $C$ into its (at most countably many) irreducible components $C_i$.
For each $i$, let $D_i$ be the set of automorphisms $ \phi$ of $\XX$ such that the image of $\XX\setminus C_i$ under $\phi$ has a nonempty intersection with $C_i$. Clearly each $D_i$ is open, and it is also dense:
given $h\in\Aut(\XX)$ not in $D_i$, let $c\in C_i$, $d=h(c)\in C_i$ and $\gamma\in\Aut(\XX)$ mapping $\bar{x}_0$ to $d$. 
Now, since the assumption in Equation \ref{spans} implies that the tangent space at $\bar x_0$ is spanned by complete fields, there exists a complete field $\alpha$ from the collection $\overline{S}$ such that $\gamma_*(\alpha)$ is not tangent to $C_i$. 
If $\varphi$ is the flow of $\gamma_*(\alpha)$, then $\varphi^t\circ h$ is an automorphism arbitrarily close to $h$ mapping $c$ out of $C_i$. By the Baire Category Theorem, since $\Aut(\XX)$ is a complete metric space, there exists a $\psi\in\bigcap D_i$.
By expanding $\overline{S}$ to $\overline{S}\cup\{(\psi_*\alpha,\psi_*\beta);(\alpha,\beta)\in\overline{S}\}$, we obtain a finite collection of semi-compatible fields which fail to satisfy Equation \ref{spans} in an exceptional variety of dimension strictly lower than that of $C$. 
The conclusion follows from the finite iteration of this procedure. 

By the previous lemmas, if $(\alpha,\beta)\in S$ then $(\alpha_u,\beta_v)$ and $(\alpha_v,\beta_u)$ are semi-compatible pairs in $\XX$.
We let $\overline{S}$ consist of all those pairs. We will also add two pairs to $\overline{S}$, of the form $(\phi^*\alpha_u,\phi^*\beta_v)$, where $\phi$ is an automorphism of $\XX$ (preserving a volume form, if necessary) to be specified later.
We now select an appropriate $\bar x_0=(u_0,v_0,x_0)\in\XX$ by picking any element from the complement of finitely many analytic subsets which we now describe. The first analytic subset of $\XX$ to avoid is the locus where any of the (finitely many) associated ideals $I_j$ vanish.
Note that Equation \ref{span} is in fact satisfied everywhere on $X$ except an analytic variety $C$: the second closed set in $\XX$ to avoid is the preimage of $C$ under the projection.
Finally, we avoid $u=0,v=0$ and $d_{x_0}f=0$. In short, we pick a  $\bar x_0=(u_0,v_0,x_0)\in\XX$ with $u_0\neq 0,v_0\neq 0,d_{x_0}f\neq 0$, such that Equation \ref{span} is satisfied at $x_0$, and such that none of the ideals $I_j(\bar x_0)$ vanish.
Because of this last condition, it will suffice to show that $\{A(\bar x_0)\wedge B(\bar x_0); (A,B)\in \overline{S}\}$ spans $\wedge^2(T_{x_0} X)$.

Consider $\pi:\XX\to X\times \C_u$, 
which at $\bar x_0$ induces an isomorphism $d_{\bar x_0} \pi :T_{\bar x_0} \XX\to T_{x_0}X\times T_{u_0}\C$. Denote by $\partial_u=\frac{\partial}{\partial u}$ the basis of $T_{u_0}\C$, and consider 
\[
P:\wedge^2(T_{\bar x_0}\XX)\to  \wedge^2 (T_{x_0}X\oplus \left<\partial_u\right>)=\wedge^2(T_{x_0}X)\oplus(T_{x_0}X\otimes\left<\partial_u\right>).
\]
Since $P$ is a linear isomorphism, it now suffices to show that the direct sum on the right-hand side equals $P(\Lambda)$, where \[\Lambda=\Span\{A(\bar x_0)\wedge B(\bar x_0);(A,B)\in \overline{S}\}.\]
We will prove (i) that $\wedge^2(T_{x_0}X)\subseteq P(\Lambda)$, and (ii) that
$T_{x_0}X\otimes\left<\partial_u\right>\subseteq P(\Lambda)$.

Let us first show (i). Let $\alpha( x_0)\wedge \beta( x_0)\in\wedge^2(T_{x_0}X)$. Since $\{\alpha(x_0)\wedge \beta(x_0)\}_{(\alpha,\beta)\in S}$ spans $\wedge^2(T_{x_0} X)$, we can assume that $(\alpha,\beta)$ is a pair of vector fields lying in $S$ 
(we will often omit to indicate the point $x_0$ at which these fields are evaluated). Then $(\alpha_u, \beta_v)\in \overline{S}$, so $P(\Lambda)$ contains
\begin{equation}\label{Lambda1}
P(\alpha_u\wedge\beta_v)=P((v\tilde{\alpha}+\alpha(f)\partial_u)\wedge(u\tilde{\beta}+\beta(f)\partial_v))=uv(\alpha\wedge\beta)-u\alpha(f)(\beta\wedge\partial_u). 
\end{equation}
At the point $\bar x_0$, we have assumed that $u$ and $v$ are both nonzero. If $\alpha(f)$ happens to vanish at $x_0$, then $\alpha(x_0)\wedge\beta(x_0)$ is in $P(\Lambda)$, as desired. Otherwise, consider the vector field $(u-u_0)\alpha_v$ on $\XX$. 
Since $\alpha_v$ is complete and $(u-u_0)$ lies in the kernel of $\alpha_v$, $(u-u_0)\alpha_v$ is a complete (and $\bar\omega$-divergence free) vector field on $\XX$. 
Quite generally one can compute, in local coordinates for example, that the flow at time $1$ of the field $g\Theta$, where $\Theta\in\CVF(M)$ and $g\in\Ker(\Theta)$ with $g(p)=0$, is a map $\phi$ whose derivative at $p\in M$ is given by:
\[
w\mapsto w+d_pg(w)\Theta(p)\quad w\in T_pM.
\]
Therefore, for a vector field $\mu\in\VF(M)$, we have
\begin{equation}\label{explicitvf}
(\phi^{-1})^*(\mu)(p)=(d_{p}\phi)(\mu(p))=\mu(p)+\mu(g)(p)\Theta(p). 
\end{equation}
Apply this in the case of $M=\XX$, $p=\bar{x}_0$, $\Theta=\alpha_v$ and $g=u-u_0$. For the vector fields 
$\mu=\beta_v$, this equals $\beta_v$; for $\mu=\alpha_u$, it equals $\alpha_u+\alpha(f)\alpha_v$.
Hence, if we add $((\phi^{-1})^*\alpha_u,(\phi^{-1})^*\beta_v)$ to $\overline{S}$, we obtain that $P(\Lambda)$ contains 
\[
P((\phi^{-1})^*\alpha_u\wedge(\phi^{-1})^*\beta_v-\alpha_u\wedge\beta_v)=P(\alpha(f)\alpha_v\wedge\beta_v)=\alpha(f)u^2(\alpha\wedge\beta).
\]

We now show (ii). It will be useful to distinguish elements in $T_{x_0}X$ according to whether they belong to $K=\Ker(d_{x_0}f)$ or not. 
Since we have assumed $d_{x_0}f\neq 0$, $T_{x_0}X$ splits as $K\oplus V$, where $V$ is a vector space of dimension $1$, which may be spanned by some $\xi$ satisfying $d_{x_0}f(\xi)=\xi(f)=1$.
The isomorphism is given by the unique decomposition $v=(v-v(f)\xi)+v(f)\xi$. This induces another splitting 
\begin{align*}
\wedge^2(T_{x_0}X)&\to \wedge^2(K)\oplus(K\otimes V)\\
\alpha\wedge\beta&\mapsto (\alpha-\alpha(f)\xi)\wedge(\beta-\beta(f)\xi)+(\alpha(f)\beta-\beta(f)\alpha)\wedge\xi.
\end{align*}
Since the left-hand side is generated by $\{\alpha\wedge\beta;(\alpha,\beta)\in S\}$, $K\otimes V$ is generated by $\{(\alpha(f)\beta-\beta(f)\alpha)\wedge\xi;(\alpha,\beta)\in S\}$, and therefore $K$ by $\{\alpha(f)\beta-\beta(f)\alpha;(\alpha,\beta)\in S\}$. 
Consider Equation \ref{Lambda1} and subtract $P(\alpha_v\wedge\beta_u)=uv(\alpha\wedge\beta)+u\beta(f)(\alpha\wedge\partial_u)$: recalling that $u_0\neq 0$, we see that \begin{equation}\label{Kotimesu}
                                                                                                       \{u(\beta(f)\alpha-\alpha(f)\beta)\wedge\partial_u;(\alpha,\beta)\in S\}=K\otimes \left<\partial_u\right>\subset P(\Lambda).
                                                                                                       \end{equation}
It remains to show that $V\otimes    \left<\partial_u\right>\subset P(\Lambda)$. By linearity, since $V$ is of dimension $1$, it suffices to find a single element in $P(\Lambda)\cap (V\otimes \left<\partial_u\right>)$. 
In fact since we have already proven (i), it suffices to find an element in $P_2(\Lambda)\cap (V\otimes \left<\partial_u\right>)$, 
where $P_2$ is the second component of the map $P$. 
%
If it were the case that for some pair $(\alpha,\beta)\in S$ both $\alpha(f)$ and                                                                                                    
$\beta(f)$ are nonzero at $x_0$,  then by Equation \ref{Lambda1} $-u\alpha(f)\beta\wedge\partial_u$ is such an intersection element.
In the other case, there is at least a pair $(\alpha,\beta)\in S$ for which $\alpha(f)(x_0)=0$ and both $\beta(f)(x_0)\neq 0$ and $\alpha(x_0)\neq 0$, for otherwise the spanning condition implied by Equation \ref{Kotimesu} would fail to be satisfied.
As in the proof of (i), we will add to $\overline{S}$ the pair $(\phi^*(\alpha_u),\phi^*(\beta_v))$, where $\phi$ is the time $1$ map of the flow of the complete (volume-preserving) field $\Theta={g}(x)(u\partial_u-v\partial_v)$, and $g\in\hol(X)$ vanishes at $x_0$.
By Equation \ref{explicitvf}, we have that \[
                                          \phi^*(\alpha_u)=\alpha_u+\alpha_u({g})\Theta=v\alpha+\alpha(f)\partial_u+v\alpha(g)(u\partial_u-v\partial_v)
                                           \]
which by assumption simplifies to \[\phi^*(\alpha_u)=v\alpha+uv\alpha(g)\partial_u-v^2\alpha(g)\partial_v.\]
Similarly we have \[
\phi^*(\beta_v)=u\beta+u^2\beta(g)\partial_u+(\beta(f)-uv\beta(g))\partial_v.
\]
Hence \[
P_2(\phi^*(\alpha_u)\wedge\phi^*(\beta_v))=u^2v\beta(g)\alpha\wedge\partial_u-u^2v\alpha(g)\beta\wedge\partial_u.       
      \]
By assumption, the first summand lies in $K\otimes \left<\partial_u\right>$, which we have already shown to be contained in $P(\Lambda)$. 
Since $\beta(f)\neq 0$, the second summand, if nonzero, lies in $P_2(\Lambda)\cap (V\otimes \left<\partial_u\right>)$. But it is clear that we may find a $g\in\hol(X)$ such that $\alpha(g)(x_0)\neq 0$.
\end{proof}
Finally, we show how the transitivity requirement for the previous proposition can be inherited from the base space $X$.
We say that a Stein manifold $X$ is \textit{holomorphically (volume) flexible} if the complete (volume-preserving) vector fields span the tangent space $T_xX$ at every $x\in X$ (see \cite[\S 6]{5auth}).
Clearly, a manifold $X$ is holomorphically (volume) flexible if one point $x\in X$ is, and $\Aut(X)$ (resp. $\Aut_\omega(X)$) acts transitively. 
Moreover, holomorphic (volume) flexibility implies the the transitive action of $\Aut(X)$ (resp. $\Aut_\omega(X)$) on $X$. 
\begin{lemma}\label{L5}
If $X$ is holomorphically flexible, then $\Aut(\XX)$  acts transitively. Moreover, if $X$ is holomorphically volume flexible at a point $x\in X$ and $\Aut_\omega(X)$ acts transitively, then $\Aut_{\bar\omega}(\XX)$ acts transitively.
\end{lemma}
\begin{proof}
For simplicity we prove the first statement: the second is proven in an exactly analogous manner. Let $\bar{x}_0=(u_0,v_0,x_0)\in\XX$ with $u_0v_0\neq 0$, and let us determine the orbit of $ \bar{x}_0$ under $ \Aut(\XX)$.
Given $\Theta\in\VF(X)$, by  Equation \ref{flowoflift} we have, for each $t$, an automorphism of $\XX$ of the form
\begin{equation}\label{orbit}
(u,v,x)\mapsto (u+tg(x,tv),v,\phi^{tv}(x)).
\end{equation}
The orbit of $\bar{x}_0$ must hence contain the hypersurface $\{ v=v_0 \}\subset\XX$ (because $\Aut(X)$ acts transitively on $X$), and analogously, since $u_0\neq 0$, the orbit contains $\{u=u_0\}\subset\XX$. 
Let $(u_1,v_1,x_1)\in \XX$ be another point with $u_1v_1\neq 0$. 
Note that the nonconstant function $f:X\to\C$ can omit at most one value $\xi$. 
Indeed, by flexibility  there is a complete vector field which at $x_0$ points in a direction where $f$ is not constant; 
precomposition with its flow map at $x_0$ gives an entire function which must omit at most one value.
Of course $\xi$ cannot be $0$, and by definition neither $u_0v_0$ nor $u_1v_1$.
Follow the orbit of $\bar x_0$ along the hypersurface $\{ u=u_0 \}\cap \XX$ until  $(u_0,v_1,x')$, then along  $\{ v=v_1 \}\cap \XX$ until $(u_1,v_1,x_1)$ (if $\xi=u_0v_1$ replace $v_1$ by $2v_1$).
So the orbit contains all points $(u,v,x)\in\XX$ with $uv\neq 0$ and by Equation \ref{orbit} also those with either $u$  or $v$ nonzero.
Consider now a point of the form $(0,0,x_0)\in\XX$. Since $x_0\in X_0$ and $X_0$ is reduced, $d_{x_0}f\neq 0$, so there is a tangent vector evaluating to a nonzero number, which since $X$ is flexible can be taken to be of the form $\Theta(x_0)$ for  a complete field $\Theta$. 
By lifting $\Theta$ we obtain an automorphism of $\XX$ of the form $(u,v,x)\mapsto (g(0,x_0),0,x_0)$. Since \[
g(0,x_0)=\lim_{t\to 0}\frac{f(\phi^t(x_0))-f(\phi^0(x_0))}{t}=(f\circ \phi)'(0)=d_{x_0}f(\Theta(x_0))\neq 0,
\]
this automorphism moves $(0,0,x_0)$ to a point of nonzero $u$ coordinate, and we are done.
\end{proof}
In particular, by the Andersén-Lempert theorem (see e.g. \cite[\S 2.B]{KK-state}), the assumptions hold if $X$ has the $\omega$-VDP and is of dimension $n\geq 2$.
\section{Examples or Applications}\label{sec:last}
The following theorem summarizes the previous discussion and gives conditions under which the suspension over a manifold has a VDP.
\begin{theorem}\label{th:4.1}
Let $X$ be a Stein manifold of dimension $n\geq 2$ such that $H^n(X)=H^{n-1}(X)=0$.
Let $\omega$ be a volume form on $X$ and suppose that 
$\Aut_\omega(X)$ acts transitively.
Assume that there is a finite collection $S$ of semi-compatible pairs $(\alpha,\beta)$ of volume-preserving vector fields such that for some $x_0\in X$, $\{\alpha(x_0)\wedge\beta(x_0);(\alpha,\beta)\in S\}$ spans $\wedge^2T_{x_0}X$.
Let $f:X\to \C$ be a nonconstant holomorphic function with smooth reduced zero fiber $X_0$ and $\tilde{H}^{n-2}(X_0)=0$. 
Then the suspension $\XX\subset \C^2_{u,v}\times X$ of $X$ along $f$ has the VDP with respect to a natural volume form $\bar{\omega}$ satisfying
$d(uv-f)\wedge\bar{\omega}=(du\wedge dv \wedge \omega)|_{\XX}$.
\end{theorem}
\begin{proof}
The spanning condition on $\wedge^2 TX$ implies holomorphic volume flexibility at $x_0$. So by Lemma \ref{L5}, $\Aut_{\bar\omega}(\XX)$ acts transitively,
and therefore Theorem \ref{Jo} may be applied.
By assumption and Proposition \ref{P2}, the topological condition of Proposition \ref{crit} is also trivially satisfied.
\end{proof}
\begin{corollary}
Let $n\geq 1$ and $f\in\hol(\C^n)$ be a nonconstant holomorphic function with smooth reduced zero fiber $X_0$, such that $\tilde{H}^{n-2}(X_0)=0$ if $n\geq 2$.
Then the hypersurface $\overline{\C^n_f}=\{uv=f(z_1,\dots,z_n)\}\subset\C^{n+2}$ has the volume density property with respect to the form $\bar{\omega}$ satisfying $d(uv-f)\wedge\bar{\omega}=du\wedge dv \wedge dz_1\wedge\dots\wedge dz_n$.
\end{corollary}
\begin{proof}
If $n\geq 2$ this follows immediately from the previous theorem, since in $\C^n$ the standards derivations $\partial_{z_j}$ generate $\wedge^2 TX$.
If $n=1$, there are no semi-compatible pairs on $\C$, but it is possible to show the VDP directly. 
Given $\Theta\in\VF_\omega(\overline{\C_f})$ and a compact $K$ of $\overline{\C_f}$, we must find a finite Lie combination of volume-preserving fields approximating $\Theta$ on $K$.
Because of this approximation, we can reduce to the algebraic case, which is treated  in \cite{KK-Zeit} by means of explicit calculation of Lie brackets of the known complete fields $\Theta_u,\Theta_v$, and $h(u\partial_u-v\partial_v)$.
\end{proof}

Let $\phi:\C^{n-1}\to\C^{n}$ be a proper holomorphic
embedding, and consider the closed subset $Z=\phi(\C^{n-1})\subset \C^{n}$.
It is a standard result that every multiplicative Cousin distribution in $\C^n$ is solvable, since $H^2(\C^n,\mathbb{Z})=0$. This implies that the divisor associated to $Z$ is principal: in other words, there exists
a  holomorphic function $f$ on $\C^{n}$ vanishing precisely on $Z$ and such that $df\neq 0$ on $Z$. We may therefore consider the suspension $\overline{\C^n_f}$ of $\C^{n}$ along $f$, which according to the above corollary must have the volume density property.
The significance of this lies in the existence of non-straightenable embeddings. Recall that a proper holomorphic embedding $\phi:\C^{k}\hookrightarrow\C^n$ is said to be \textit{holomorphically straightenable} 
if there exists an automorphism $\alpha$ of $\C^n$ such that $\alpha(\phi(\C^k))=\C^k\times \{0\}^{n-k}$. 
The existence of non-tame sets in $\C^n$, combined with an interpolation theorem, implies that there exists
for each $k<n$ non-straightenable proper holomorphic embeddings $\phi:\C^{k}\hookrightarrow\C^n$, see 
\cite{Finterpol}. Note that proper algebraic embeddings are the holomorphic analogue of polynomial embeddings,
and that the ``classical'' algebraic situation is in sharp contrast to the holomorphic one: 
for every $n>2k+1$ polynomial embeddings $\phi:\C^{k}\hookrightarrow\C^n$ are algebraically straightenable (see \cite{Kaliman-AM}),  the case of real codimension $2$ remaining notoriously open.

If the embedding $\phi$ is straightenable, it is clear that $\overline{\C^n_f}$ is trivially biholomorphic to $\C^{n+1}$, and a calculation shows that the form $\bar\omega$ is the standard one.
So the result says something new only if $\phi$ is  non-straightenable.
Indeed, it is unknown whether $\overline{\C^n_f}$ is biholomorphic to $\C^{n+1}$.
However, $\overline{\C^n_f}\times\C$ is biholomorphic to $\C^{n+2}$ (see \cite{Derksen-Kut}),
and is therefore a potential counterexample to the holomorphic version of the important Zariski Cancellation Problem:
if $X$ is a complex manifold of dimension $n$ and $X\times\C$ biholomorphic to $\C^{n+1}$, does it follow  that $X$ is biholomorphic to $\C^n$?

Moreover, $\overline{\C^n_f}$ is diffeomorphic to complex affine space. This is best shown in the algebraic language of modifications, as follows. Given a triple $(X,D,C)$ consisting of a Stein manifold $X$, a smooth reduced analytic divisor $D$, and a proper closed complex submanifold $C$ of $D$,
it is possible construct the pseudo-affine modification of $X$ along $D$ with center $C$, denoted $\XX$. It is the result of blowing up $X$ along $C$ and deleting the proper transform of $D$.
We refer the interested reader to \cite{KZaffine} for a general discussion.
In our situation we take $X=\C^n\times \C_u$, $D=\C^n\times \{0 \}$, and $C=Z\times \{ 0 \}=\phi(\C^{n-1})\times \{0\}$: it can be shown that in this case $\XX$ is biholomorphic to $\overline{\C^n_f}$ (see \cite{KZaffine}).
We now invoke a general result giving sufficient conditions for a pseudo-affine modification to be diffeomorphic to affine space: since $Z$ is contractible, Proposition 5.10 from \cite{KK-Zeit} is directly applicable, and therefore the following holds:
\begin{corollary}
If $\phi:\C^{n-1}\to\C^{n}$ is a proper holomorphic embedding, then the suspension $\overline{\C^n_f}$ along the function $f$ defining the subvariety $\phi(\C^{n-1})$, is diffeomorphic to $\C^{n+1}$
and has the volume density property with respect to a natural volume form $\bar\omega$. Moreover $\overline{\C^n_f}\times \C$ is biholomorphic to $\C^{n+2}$, and is therefore a potential counterexample to the Zariski Cancellation Problem if $\phi$ is not straightenable.
\end{corollary}
Recall a conjecture of A. T\'{o}th and Varolin \cite{Toth-Var} asking whether a complex manifold which is diffeomorphic to $\C^n$ and has the density property must be biholomorphic to $\C^n$. 
It is also unknown whether there are contractible Stein manifolds with the volume density property which are not biholomorphic to $\C^n$, and our construction provides a new potential counterexample. As pointed out in Section \ref{sec:intro}, this is to our knowledge the first non-algebraic one.

To conclude, we give another example of an application.
Consider a proper holomorphic embedding $\mathbb{D}\hookrightarrow\C^2_{x,y}$ (that this exists is a classical theorem of K. Kasahara and T. Nishino, see e.g. \cite{Stehle}), and let $f$ generate the ideal of functions vanishing on the embedded disk, as above. 
Then $M=\overline{\C^2_f}\subset\C^2_{u,v}\times\C^2_{x,y}$ admits a $\C^*$-action, namely
\[
\lambda\mapsto(\lambda u,\lambda^{-1}v, x,y),
\]
whose fixed point set is biholomorphic to $\mathbb{D}$.
Therefore, the action cannot be linearizable, i.e., there is no holomorphic change of coordinates after which the action is linear. 
Recall the problem of linearization of holomorphic $\C^*$-actions on $\C^k$ (see e.g. \cite{Derksen-Kut}): for $k=2$, every action is linearizable; there are counterexamples for $k\geq 4$; and the problem remains open for $k=3$.
If $M$ is biholomorphic to $\C^3$, there would be a negative answer. Otherwise, it resolves in the negative the T\'{o}th-Varolin conjecture mentionned above.
By a result of J. Globevnik \cite{Globevnik}, it is also possible to embed arbitrary small perturbations of a polydisc in $\C^n$ for any $n\geq 1$ into $\C^{n+1}$; by the same argument, we obtain for any $n\geq 3$, non-algebraic manifolds that are diffeomorphic to $\C^{n}$ with the volume density property.

\newcommand{\etalchar}[1]{$^{#1}$}
\providecommand{\bysame}{\leavevmode\hbox to3em{\hrulefill}\thinspace}
\providecommand{\MR}{\relax\ifhmode\unskip\space\fi MR }
\providecommand{\MRhref}[2]{%
  \href{http://www.ams.org/mathscinet-getitem?mr=#1}{#2}
}
\providecommand{\href}[2]{#2}

\end{document}